\documentclass{amsart}

\usepackage[utf8]{inputenc}

\usepackage{amsmath,amssymb,latexsym,amsthm}
\usepackage{hyperref}
\usepackage[nobysame,alphabetic,initials]{amsrefs}

\usepackage{xcolor}

\newtheorem{thm}{Theorem}
\newtheorem{lemma}{Lemma}

\theoremstyle{definition}

\allowdisplaybreaks

\newcommand{\R}{\mathbb{R}}
\newcommand{\EE}{\mathbb{E}}
\newcommand{\NN}{\mathbb{N}}
\newcommand{\Var}{{\rm Var}\, }

\newcommand{\conv}{{\rm conv }\,}
\newcommand{\vis}{{\rm Vis }\,}
\newcommand{\dx}{{\rm d}}
\newcommand{\pr}{\mathbb P}
\newcommand{\inti}{{\rm int }\,}
\newcommand{\bd}{\mathrm{bd}\,}

\title[A central limit theorem for random disc-polygons]{A central limit theorem for random disc-polygons in smooth convex discs}
\author[F. Fodor]{Ferenc Fodor}
\address{Department of Geometry, Bolyai Institute, University of Szeged, Aradi v\'ertan\'uk tere 1, Szeged, H-6720, Hungary}
\email{fodorf@math.u-szeged.hu}

\author[D. I. Papv\'ari]{D\'aniel I. Papv\'ari}
\address{Bolyai Institute, University of Szeged, Aradi v\'ertan\'uk tere 1, Szeged, H-6720, Hungary}
\email{papvari.daniel.istvan@stud.u-szeged.hu}


\begin{document}
	
	\begin{abstract}
		In this paper we prove a quantitative central limit theorem for the  area of uniform random disc-polygons in smooth convex discs whose boundary is $C^2_+$. We use Stein's method and the asymptotic lower bound for the variance of the area proved by Fodor, Gr\"unfelder and V\'igh \cite{FGV22}.
	\end{abstract}
	
	\maketitle
	
	\section{Introduction and results}
	The study of the asymptotic behaviour of random polytopes is a venerable topic in stochastic geometry going back to the ground-breaking papers of R\'enyi and Sulanke \cites{RS63, RS64}. Several models have been considered, of which the most investigated is probably the one where the random polytope $K_n$ arises as the convex hull of $n$ i.i.d. random points from a convex body selected according to the uniform distribution. For a comprehensive survey of results on this and other models we refer to the papers by B\'ar\'any \cite{Bar08}, Reitzner \cite{R10} and Schneider \cite{Sch18} and for the references therein.     
	
	Central limit theorems have been proved recently in various models for diverse quantities associated with random polytopes. We only mention a few such results that are most closely related to our topic. Reitzner \cite{Rei05} proved an asymptotic lower bound for the variance of the missed volume $V(K\setminus K_n)$ (also for the number of $i$-dimensional faces $f_i(K_n)$ of $K_n$) when $K$ has $C^2_+$ smooth boundary. With the help of this lower bound he showed that $V(K\setminus K_n)$ (and also $f_0(K_n)$) satisfy a central limit theorem. His method used an extra randomization through Poisson polytopes. With similar methods, B\'ar\'any and Reitzner \cite{BR10} proved central limit theorems for the same quantities in the case when $K$ is a polytope.
	Using stabilizing functionals, Lachi\`eze-Rey, Schulte and Yukich \cite{LSY19} established CLTs for all intrinsic volumes of $K\setminus K_n$ for $K$ with $C^2_+$ boundary. Th\"ale, Turchi and Wespi \cite{TTW18} proved independently central limit theorems for all intrinsic volumes using floating bodies and Stein's method. More information and further references to recent developments regarding limit theorems in other models can be found, for example, in Besau, Rosen and Th\"ale \cite{BRT21}, and Th\"ale \cite{T18}.  
	
	There have been several papers dedicated recently to approximations of convex bodies by various generalizations of random polytopes. One such model uses intersections of congruent closed balls to generate a hull, and the resulting notion of convexity is often called spindle or ball convexity. In this paper, we will use this notion of convexity in the planar $\R^2$ setting. Precise definitions are the following.

	%Subsequently, we always assume that all random variables are defined on a common probability field $(\Omega, \mathcal{A},\pr)$, where  $\EE(\cdot)$ and $\Var(\cdot)$ denote expectation and variance, respectively, according to  $\pr$.
	
	Let $r>0$ be fixed. For $x,y\in\R^2$ with $|x-y|\leq 2r$, let $[x,y]_r$ denote the intersection of all radius $r$ closed circular discs that contain both $x$ and $y$. The set $[x,y]_r$ is called the $r$-spindle of $x$ and $y$. A compact set $K\subset\R^2$ is called spindle convex with radius $r$ (or $r$-spindle convex) if for any $x,y\in K$ it holds that $[x,y]_r\subset K$. This also means that the shorter arc of any circle of radius at least $r$ incident with $x$ and $y$ is contained in $K$. We call the intersection of finitely many radius $r$ closed circular discs a disc-polygon (of radius $r$), or an $r$-disc-polygon for short, which itself is spindle convex with radius $r$. 
	Let $S\subset\R^2$ be a set that is contained in a circle of radius $r$. The intersection of all closed radius $r$ circular discs that contain $S$ is called the (closed) $r$-spindle convex hull of $S$, which we denote by $[S]_r$. In particular, if $S\subset K$, where $K$ is $r$-spindle convex, then $[S]_r\subset K$. 
	
	A particularly important class of spindle convex sets are those (linearly) convex discs whose boundary $\bd K$ is of class $C^2_+$, that is, twice continuously differentiable with positive curvature. Let $K\subset\R^2$ be a convex disc (compact, convex set with non-empty interior) whose boundary $\bd K$ is of class $C_+^2$. Let $r_M=\max 1/\kappa (x)$ for $x\in \bd K$, where $\kappa(x)$ is the curvature of $\bd K$ at $x$. It is known that $K$ is $r$-spindle convex for any $r\geq r_M$, cf. \cite{Schneider}. 
	
	For more information on geometric properties of spindle convex sets we refer to Bezdek et al \cite{BL07}, and Martini, Montejano, Oliveros \cite{MMO19} and the references therein.
	
	In this paper we study the following probability model. Let $K$ be a convex disc with $C^2_+$ boundary and $r>r_M$. 
	Let $n\geq 2$, and consider $n$ 
	i.i.d. random points $X_1,\ldots,X_n$ from $K$  selected according to the uniform probability distribution. Let $K_n^r=[X_1,\ldots,X_n]_r$, which is a (uniform) random $r$-disc-polygon. Since $K$ is $r$-spindle convex, $K_n^r\subset K$. We denote by $A(K_n^r)$ the area of $K_n^r$.
	
	The asymptotic expectation of the random variables $f_0(K_n^r)$ and $A(K_n^r)$ were determined in \cite{FKV14}, where the following theorem was proved.
	
	\begin{thm}[\cite{FKV14}]\label{thm:FKVcsucs}
		Let $K$ be a convex disc whose boundary is of class $C_+^2$. Then for $r>r_M$, it holds that
		\[\lim_{n\to\infty}\EE\big[A(K\setminus K_n^r)\big]n^{\frac 23}=\sqrt[3]{\frac{2 A^2(K)}{3}}\Gamma\left(\frac53\right)\int_{\bd K}\left(\kappa(x)-\frac 1r\right)^{\frac 13}\dx x.\]
	\end{thm}
	In the above formula, $\Gamma(\cdot)$ is Euler's gamma function, and integration is with respect to the arc-length on $\bd K$. 
	Theorem~\ref{thm:FKVcsucs} is a generalization, as $r\to\infty$, of the classical  results of R\'enyi and Sulanke \cite{RS63} regarding the linear convex hull of the random points $X_1, \ldots, X_n$. 
	
	For convenience, in the foregoing we use the following symbols to denote orders of magnitude.
	If $(a_n)_{n\in \NN}$ and $(b_n)_{n\in \NN}$ are sequences with the property that there exists a constant $c\in(0,\infty)$ such that for all $n$ (or, equivalently, for all $n$ greater than some threshold $n_0$) $a_n\leq c b_n$ is satisfied, then we write $a_n\ll b_n$. 
	If $a_n\ll b_n$ and $b_n\ll a_n$, then this fact is indicated by the $a_n\approx b_n$ notation. We note that, in general, $a_n\approx b_n$ does not necessarily mean the asymptotic equality of $(a_n)$ and $(b_n)$, as the corresponding constants may be different. 
	
	It is usually more difficult to obtain results about higher moments of random variables associated with random (disc-) polygons than expectations. Fodor and V\'igh \cite{FV18} proved asymptotic upper bounds for the area $A(K_n^r)$. 
	
	\begin{thm}[{\cite{FV18}}]\label{thm:varupp}
		Let $K$ be a convex disc whose boundary is of class $C_+^2$. Then for  $r>r_M$, it holds that
		%\[\Var[f_0(K_n^r)]\ll n^{\frac 13}\quad\text{and}\quad
		\[\Var[A(K_n^r)]\ll n^{-\frac 53},\]
		where the implied constant depends only on $K$ and $r$.
	\end{thm}
	
	Using Theorem~\ref{thm:varupp}, one can prove the strong law of large numbers by standard methods, see \cite[Theorem 5 on p. 1145]{FV18}.
	
	%\begin{thm}[{\cite{FV18}}] Under the same assumptions as in Theorem~\ref{thm:varupp}, it holds with probability $1$ that
	%	\[\lim_{n\to\infty}f_0(K_n^r)n^{-\frac 13}=\sqrt[3]{\frac{2}{3 A(K)}}\Gamma\left(\frac53\right)\int_{\bd K}\left(\kappa(x)-\frac 1r\right)^{\frac 13}\dx x,\]
	%	and
	%	\[\lim_{n\to\infty}A(K\setminus K_n^r)n^{\frac 23}=\sqrt[3]{\frac{2 A^2(K)}{3}}\Gamma\left(\frac53\right)\int_{\bd K}\left(\kappa(x)-\frac 1r\right)^{\frac 13}\dx x.\]
	%\end{thm}
	
	Based on an argument of Reitzner \cite{Rei05}, Fodor, Gr\"unfelder and V\'igh \cite{FGV22} proved matching asymptotic lower bounds for the area $A(K_n^r)$ (and also for the number of vertices).
	
	\begin{thm}[{\cite{FGV22}}]\label{thm:varlow}
		Under the same assumptions as in Theorem~\ref{thm:varupp}, it holds that
		\[\Var[A(K_n^r)]\approx n^{-\frac 53}.\]
	\end{thm}
	Using the asymptotic lower bound on the variance of the area in Theorem~\ref{thm:varlow}, we prove a quantitative central limit theorem for  $A(K_n^r)$ as $n\to\infty$. Our argument uses the normal approximation bound proved by Chatterjee \cite{Cha08} and Lachi\`eze-Rey and Peccati \cite{LRP17} that originated from  Stein's method \cite{Ste86}. We note that no central limit theorem is available currently for any other quantity in the spindle convex model.
	
	The Wasserstein distance of two random variables $X$ and $Y$ defined on the same probability space is 
	\[\dx_{W}(X,Y):=\sup_{h\in {\rm Lip}_1}\big|\EE[h(X)]-\EE[h(Y)]\big|,\]
	where ${\rm Lip}_1$ denotes the set of all Lipschitz continuous functions $h:\R\to\R$ with Lipschitz constant at most $1$. 
	The Wasserstein distance, in fact, defines a metric on (equivalence classes of) random variables on a probability space. Therefore, one can use it to define the convergence of sequences of random variables. It is known that convergence w.r.t. Wasserstein distance implies weak convergence (convergence in distribution), see, for example \cite[Ch. 6]{V09}. In particular, if $G$ is a standard normal random variable, and $(W_n)_{n\in\NN}$ is a sequence of centred random variables with finite second moments for which
	\[\lim_{n\to \infty}\dx_W\left(\frac{W_n}{\sqrt{\Var(W_n)}},G\right)=0,\]
	then $W_n/\sqrt{\Var(W_n)}\xrightarrow{\mathcal{D}}G$, where $\xrightarrow{\mathcal{D}}$ denotes convergence in distribution.

	Our argument is based on the work of Th\"ale, Turchi and Wespi \cite{TTW18}. Using estimates for floating bodies and general normal approximation bounds they gave a short and transparent proof of a central limit theorem for intrinsic volumes of classical random polytopes in smooth convex bodies, which we state here only for the case of volume. 
	\begin{thm}[{\cite{TTW18}}]%\label{thm:ttw18}
		Let $K\subset\R^d$, $d\geq 2$ be a convex body with $C_+^2$ smooth boundary. Then
		\[\dx_{W}\left(\frac{V_d(K_n)-\EE[V_d(K_n)]}{\sqrt{\Var[V_d(K_n)]}},G\right)\ll n^{-\frac 12+\frac{1}{d+1}}(\log n)^{3+\frac{2}{d+1}},\]
		where $K_n$ is the convex hull of $n\geq d+1$ i.i.d. random points that are uniformly distributed in $K$ and $G$ is a standard normal random variable.
	\end{thm}
	
	Our main result is the following theorem for the spindle convex case in the plane.
	
	\begin{thm}\label{thm:fotetel}
		Let $K$ be a convex disc with $C^2_+$ boundary. Then for any $r>r_M$ it holds that
		\begin{equation}\label{eq:wass}
			\dx_{W}\left(\frac{A(K_n^r)-\EE [A(K_n^r)]}{\sqrt{\Var [A(K_n^r)]}},G\right)\ll n^{-\frac 16}(\log n)^{3+\frac 23}. 
		\end{equation}
	\end{thm}
	We note that the order of magnitude in \eqref{eq:wass} is most likely not optimal.
	
	The rest of the paper is organized as follows. In Section~\ref{sec:geom-tools} we collect the necessary geometric tools for the proof. Section~\ref{sec:normal} contains a (very) short summary of the specific normal approximation methods we use. We prove Theorem~\ref{thm:fotetel} in Section~\ref{sec:proof}.
	
	\section{Geometric tools}\label{sec:geom-tools}
	
	We will use the so-called floating body in our arguments, which was introduced independently in \cite{BL88} by B\'ar\'any and Larman, and in \cite{SW90} by Sch\"utt and Werner. 
	%In the former, it was used to study the asymptotic behaviour of random polytopes. 
	Let $K\subset \R^2$ be a convex disc, $t>0$ (we always assume that $t$ is sufficiently small), and $H$ a closed half-plane.
	Let $v:K\to \R$ be defined as
	\[v(x)=\min\big\{A(K\cap H):x\in H, \; H\text{ closed half-plane}\big\}.\]
    If $A(K\cap H)=t$, then the set $C=K\cap H$ is called a (linear) $t$-cap of $K$. The level set
	\[K(v\leq t)=\{x\in K:v(x)\leq t\}\]
	is called the wet part of $K$ with parameter $t$. The closure of the complement of $K(v\leq t)$ w.r.t. $K$ is 
	\[K_{(t)}=K(v\geq t)=\{x\in K:v(x)\geq t\},\]
	which is the floating body of $K$ with parameter $t$.
	
	B\'ar\'any and Larman \cite{BL88} proved that the random polytope $K_n$ behaves asymptotically roughly as $K_{(1/n)}$, and the missing part $K\setminus K_n$ as the wet part $K\setminus K_{(1/n)}$. B\'ar\'any and Dalla \cite{BD97} showed the following lemma for the uniform distribution, and Vu \cite{Vu05} extended it to more general distributions using different methods. We only need the $d=2$ special case but the original statement is for general $d$. 
	\begin{lemma}[{\cites{BD97,Vu05}}]\label{lem:lemegy}
		Let $K_n$ be a random polygon in the convex disc $K\subset \R^2$ that is the convex hull of $n$ i.i.d. uniform random points. Then for any $\beta\in(0,\infty)$ there exists $c=c(\beta)\in(0,\infty)$ for which 
		\[\mathbb{P}(K_{(c\log n/n)}\not\subseteq K_n)\leq n^{-\beta},\quad\text{if $n$ is sufficiently large.}\]
	\end{lemma}

	For a point $z\in \bd K$ and a (suitably small) $t>0$ parameter, the visibility region of $z$ with parameter $t$ is the set of points in $K\setminus K_{(t)}$ that are clearly visible from $z$, that is, 
	\[{\rm Vis}(z,t)=\{x\in K\setminus K_{(t)}:[x,z]\cap K_{(t)}=\emptyset\},\]
	where $[x,z]$ denotes the segment with endpoints $x$ and $z$.
	
	Let $S\subseteq \R^2$ be a non-empty set. Then
	\[{\rm diam\,}(S)=\sup_{x,y\in S}\|x-y\|\]
	is the diameter of $S$.
	
	Let $K\subset \R^2$ be a convex disc whose boundary is of class $C_+^2$ and assume that $o$ is in the interior of $K$. Then there exists a constant $c=c(K)$ such that for sufficiently small  $t>0$ it holds that
	\[\sup_{z\in \bd K}A\big({\rm Vis}(z,t)\big)\leq ct.\]
	For a sketch of the proof see \cite[p. 3067]{TTW18}.
	
	It follows from the $C^2_+$ property that for each  boundary point $x\in\bd K$ there exists a unique outer unit normal $u_x\in S^1$. Moreover, for all $u\in S^1$ there exists a unique boundary point $x_u\in \bd K$ such that the outer unit normal at $x_u$ is $u$. For a unit vector $u\in S^1$ and real number $t\geq 0$, let $H=H(u,t)=\{x\in\R^d: \langle x, u \rangle=t\}$ be the hyperplane, and $H^+=H^+(u,t)=\{x\in\R^d: \langle u,x\rangle\geq t\}$ the closed half-space determined by $u$ and $t$. The vertex of the cap $C=K\cap H^+$ is the unique boundary point $x_u$, and the height $h$ is the distance of $x_u$ and $H$. We use the notation $C(x_u, h)$ for this cap. If
	\[\kappa_m=\min_{x\in \bd K}\kappa(x)\quad\text{and}\quad \kappa_M=\max_{x\in \bd K}\kappa(x),\]
	then a circle of radius $r_m=1/\kappa_M$ rolls freely in $K$, (see \cite[Section~3.2, p. 156]{Schneider}), that is, for all $x\in \bd K$ there exists a vector $p\in\R^2$ such that $x\in r_mB^2+p\subset K$. Moreover, $K$ slides freely in a circle of radius $r_M=1/\kappa_m$, meaning that for all $x'\in\bd K$ there exists $p'\in\R^2$ with $x'\in r_M \bd B^2+p'$ and $K\subset r_M B^2+p'$. The circular disc $r_M B^2+p'$ is called  a supporting disc of $K$ at $x'$. Due to the $C_+^2$ property of $\bd K$, the supporting disc is unique at each $x\in\bd K$. This also implies that $K$ is $r$-spindle convex for all $r\geq r_M$.
	
	By scaling, we may always assume that $A(K)=1$. %We note that a homothety preserves spindle convexity in the following sense: if the convex disc $K\subset\R^2$ is spindle convex with radius $r\geq r_M$, then for any $0<\lambda<1$ the set $\lambda K$ is spindle convex with radius $\lambda r$ (see \cite[Corollary~3.7, p. 47]{L}). From now on, we assume that $r=1$ and $\kappa(x)>1$ for all $x\in\bd K$.
	
	Let $\overline{B}^2$ denote the origin centred unit radius open ball.  A subset of the form $K\setminus \big(r\overline{B}^2+p\big)$ where $p\in \R^2$ is called a (radius $r$) disc-cap of $K$.  Next, we recall some notation from \cite{FKV14}.
	
	Let $x, y\in K$, $x\neq y$ be two points. The two radius $r$ circles incident with $x$ and $y$ determine two disc-caps of $K$, which we denote by $D_-(x,y)$ and $D_+(x,y)$, where $A\big(D_-(x,y)\big)\leq A\big(D_+(x,y)\big)$. We will use the shorter symbol $A_-(x,y)=A\big(D_-(x,y)\big)$ and $A_+(x,y)=A\big(D_+(x,y)\big)$, and for simplicity we omit $r$ from the notation of the caps. 
	
	Fodor, Kevei and V\'igh showed \cite[Lemma 4.3, p. 906]{FKV14} that if $\bd K$ is $C^2_+$ and $\kappa(x)> 1$ for all $x\in \bd K$, then there exists $\delta>0$ (depending only on $K$) such that for any $x,y\in \inti K$ it holds that $A_+(x,y)>\delta$. 
	
	Assume that $K$ is a convex disc with $C^2_+$ boundary such that $\kappa_m>1/r$. It is known (see \cite[Lemma~4.1, p. 905]{FKV14}) that if $D=K\setminus\big(r\overline{B}^2+p\big)$ is a non-empty disc-cap, then there exists a unique point $x_0\in \bd K\cap \bd D$ (the vertex) and a non-negative real number $h$ (the height) for which $rB^2+p=rB^2+x_0-(r+h)u_{x_0}$.
	
	Let $D$ be a disc-cap in $K$ with vertex $x_0$. For a line $e$ orthogonal to $u_{x_0}$ let  $e_+$ be its closed half-plane that contains $x_0$. Then there exists a maximal (with respect to inclusion) linear cap $C_-(D)=K\cap e_+$ that is contained in $D$, and a minimal cap $C_+(D)=e'_+\cap K$ containing $D$. It was proved in \cite{FV18} (see Claim~1, on page 1146), that there exists a constant $\hat c\in(0,1)$ depending only on $K$ and $r$ such that if the height of $D$ is sufficiently small, then
	\begin{align}
		\label{eq:szendvics}
		\hat{c}\big(C_+(D)-x_0\big)\subset C_-(D)-x_0.
	\end{align}
	
	This implies that a disc-cap can be "sandwiched" between two linear caps such that the height of the bigger cap is at most $\hat c$ times the height of the smaller cap. It also follows that the area of a disc-cap of height $h$ is of order of magnitude $h^{3/2}$ if $h$ is sufficiently small. The exact behaviour of the area of disc-caps as $h\to 0$ is described in the following limit. If $D(x_0,h)$ is a disc-cap with vertex $x_0$ and height $h$, then
	\[\lim_{h\to 0^+}A\big(D(x_0,h)\big)h^{-\frac 32}=\frac 43\sqrt{\frac{2}{\kappa(x_0)-1/r}},\]
	see \cite[Lemma~4.2, p. 905]{FKV14}. The $C^2_+$ property of $\bd K$ yields that there exist constants $\gamma>0$ and $\Gamma>0$, depending only on $K$, such that for any $x_0\in\bd K$ and sufficiently small $h$,
	\[\gamma h^{\frac 32}\leq A(C(x_0, h))\leq \Gamma h^{\frac 32}.\]
	In turn, \eqref{eq:szendvics} implies that there exist constants $\widetilde{\gamma}>0$ and $\widetilde{\Gamma}>0$, depending only on $K$ and $r$ such that for any $x_0\in\bd K$ and sufficiently small $h$, 
	\[\widetilde{\gamma} h^{\frac 32}\leq A(D(x_0, h))\leq \widetilde{\Gamma} h^{\frac 32}.\]
	
	We now introduce the $r$-spindle floating body and $r$-spindle wet part of a ($r$-spindle) convex disc $K$. Let $v_r:K\to\R$ be
	\[v_r(x)=\min\big\{A(K\cap D):x\in rS^1+p,\; p\in\R^2\big\},\]
	where $D=K\setminus \big(r\overline{B}^2+p\big)$ is a non-empty disc-cap in $K$.
	The level set of $v_r$
	\[K^r_{(t)}:=K(v_r\geq t)=\{x\in K: v_r(x)\geq t\}\]
	is called the $r$-spindle floating body of $K$ with parameter $t$. Correspondingly, the  $r$-spindle wet part with parameter $t$ is
	\[K(v_r\leq t)=\{x\in K:v_r(x)\leq t\}.\]
	We note that the $r$-spindle floating body (for any $t$) is also $r$-spindle convex as it is the intersection of radius $r$ closed circular discs.
 
	The following lemma shows that the $r$-spindle floating body of $K$ can also be sandwiched between two "classical" floating bodies.

 \begin{lemma}\label{lem:floating-szendo}
    Let $K$ be a convex disc with $C^2_+$ boundary. For any $r>r_M$ there exists a constant $c_0\in (0,1)$ that depends on $K$ and $r$, such that for sufficiently small $t>0$, the following inclusions hold
    \[K_{(t)}\subset K_{(t)}^r\subset K_{(c_0t)}.\]
 \end{lemma}

 \begin{proof}
     If $x\in \bd K_{(t)}^r$, then there exists a minimal disc-cap $D_-=K\setminus\big(r\overline{B}^2+p\big)$, for which $x\in rS^1+p$. We will say that $D_-$ \textit{lies on} $x$, or equivalently, that $D_-$ is a disc-cap \textit{through} $x$. The same is true for (linear) caps, where $x$ lies on a line (instead of $rS^1+p$). The area of the cap that we get by the "lower part" of $D_-$'s support line through the point $x$ is greater than or equal to the area of the minimal cap that lies on $x$. Hence, through a point $x\in K$, the area of the minimal cap is always smaller than the area of the minimal disc-cap. Thus, $K_{(t)}$ is always contained in $K_{(t)}^r$.
     
     Next, we need to show that there exists a constant $\tilde{c}$ for which the area of the minimal cap (w.r.t. $\tilde{c}t$) through $x$ is greater than $A(D_-)$. To see this, we will need the following: for an arbitrary $x\in K$, the area of the minimal disc-cap that lies on $x$ is at most some universal constant $c\geq1$ times the area of the minimal cap through $x$. Let the minimal cap through $x$ be denoted by $C_-$ with the vertex $y\in \bd K$. Now, consider the disc-cap whose vertex is also $y$, it lies on $x$ and supports $C_-$ and denote it by $D(y)$. The area of $D(y)$ is at least the area of the minimal disc-cap through $x$. However, $D(y)$ is contained in an enlarged version of $C_-$. Thus, the area of $D(y)$ is smaller than the area of the enlarged minimal cap, which is at most $c\cdot A(C_-)$ for some constant $c$. Because of the $C^2_+$ property of the boundary of $K$ and the choice of $r$, the constant for the enlargement is uniform for all $x$ and $t$. From this fact, it follows that there exists a constant $c_0\in(0,1)$, such that the floating body of $K$ with parameter $c_0t$ contains the $r$-spindle floating body of $K$ with parameter $t$.
 \end{proof}
 
	Using the fact that for any $X\subset K$,  the set $[X]_r$ (strictly) contains $\conv (X)$, it follows that $K_n\subset K_n^r$. By Lemma~\ref{lem:floating-szendo}, there exist positive constants $c$ and $c_1<c$ such that the following inclusions hold for the following events for sufficiently large $n$ 
	\[\{K_{(c\log n/n)}^r\not\subseteq K_n^r\}\subseteq \{K_{(c_1\log n/n)}\not\subseteq K_n^r\}\subseteq\{K_{(c_1\log n/n)}\not\subseteq K_n\},\]
	where $K_{(c\log n/n)}\subset K_{(c\log n/n)}^r\subset K_{(c_1\log n/n)}$. 
	By Lemma~\ref{lem:lemegy} (\cite[Lemma~4.2, p. 1298]{Vu05}) and the above, we obtain the following statement. 
	
	\begin{lemma}\label{lem:uszo}
		Let $K$ be a convex disc with $C^2_+$ boundary. For any $r>r_M$ and $\beta\in (0,\infty)$ there exists $c=c(\beta,r)\in (0,\infty)$ such that 
		\[\pr(K_{(c\log n/n)}^r\not\subseteq K_n^r)\leq n^{-\beta},\quad\text{if $n$ is large enough.}\]
	\end{lemma}
	
	We now introduce the $r$-spindle visibility regions. Let $z\in \bd K$ and $t>0$. The $r$-spindle visibility region of $z$ with parameter $t$ is the set of points in $K\setminus K^r_{(t)}$ that are visible along a radius $r$ circular arc from $z$ avoiding $K^r_{(t)}$ as an obstacle, that is,
	\[\vis_r(z,t)=\big\{x\in K\setminus K^r_{(t)}:\exists[\overset{{{\displaystyle \frown}}}{x,z}]_r\text{ such that }[\overset{{{\displaystyle \frown}}}{x,z}]_r\cap \inti K^r_{(t)}=\emptyset\big\},\]
	where $[\overset{{{\displaystyle \frown}}}{x,z}]_r$ denotes a shorter circular arc of radius $r$ with endpoints $x$ and $z$. We note that for any $x\neq z$ there are two such arcs.

	\begin{lemma}\label{lem:spindlevis}
		Let $K$ be a convex disc with $C^2_+$ boundary. Then there exists a constant $C$, depending only on $K$,  such that for any $r>r_M$ and sufficiently small $t>0$ it holds that
		\[\sup_{z\in \bd K}A\big(\vis_r(z,t)\big)\leq Ct.\]
	\end{lemma}
	
	\begin{proof}
		Note that $\vis_r(z,t)$ is the union of all area $t$ disc-caps that contain $z\in \bd K$. Let $D$ be such a disc-cap whose height is $c_1t^{2/3}$. Let $C_+(D)$ be a Euclidean cap containing $D$ with height $c_2t^{2/3}$, whose existence is guaranteed by \eqref{eq:szendvics}.
		
		Reitzner proved (see \cite[pp. 2149-2150]{Rei03}) that if $h>0$ is sufficiently small and $C_1(x_1,h_1)\cap C_2(x_2,h_2)\neq 0$, where $C_1(x_1,h_1),C_2(x_2,h_2)$ are two Euclidean caps whose vertices are $x_1$ and $x_2$, respectively, and whose heights satisfy $h\geq h_1\geq h_2$, then there exists a constant $\tilde{c}$ (depending only on $K$) 
		for which $C_2(x_2,h_2)\subset \tilde{c}\big(C_1(x_1,h_1)-x_1\big)+x_1$. Using this for all caps $C_+(D)$, we obtain that there exists a disc-cap $D(z,c_1t^{2/3})$ which is contained in $C_+\big(D(z,c_1t^{2/3})\big)$, and there exists a constant $C$ (depending only on $K$ and the radius $r>r_M$) such that if we blow up $C_+\big(D(z,c_1t^{2/3})\big)$ by a factor of  $C$, then the resulting disc-cap contains $\vis_r(z,t)$ and its area is of order $t$. 
	\end{proof}

	\section{Stein's method, normal approximation bounds}\label{sec:normal}
	We summarize (very briefly) the most necessary notation and statements we need for our normal approximation bound. For more information on the method we refer to the paper by Chatterjee \cite{Cha08} and Lachi\`eze-Rey, Peccati \cite{LRP17}. 
	Let $E$ be a complete, separable metric space (Polish space). In our application in the proof of Theorem~\ref{thm:fotetel} $E$ will be the interior of the convex disc $K$ in $\R^2$ with $C^2_+$ boundary. Let $X=(X_1,\ldots,X_n)$ be $n$ i.i.d.
	random variables that are elements of $E$, and let $X', X''$ be independent copies of $X$. We denote the $i$-th coordinate ($i\in \{1,\ldots, n\}$) of $X'$ and $X''$ by $X'_i$ and $X''_i$, respectively. 
	
	We say that a random vector $Z=(Z_1,\ldots,Z_n)$ is a recombination of $\{X,X',X''\}$ if $Z_i\in\{X_i,X'_i,X''_i\}$ for all $i\in\{1,\ldots,n\}$.
	
	Let $f:\cup_{k=1}^{n}E^k\to \R$ be a measurable and symmetric function acting on point configurations of at most $n\in\NN$ points in $E$. For $x=(x_1,x_2,\ldots,x_n)\in E^n$ and $i\in\{1,\ldots,n\}$ we denote by
	\[x^{\neg i}=(x_1,x_2,\ldots,x_{i-1},x_{i+1},\ldots,x_n)\in E^{n-1}\]
	the vector that we get from $x$ by removing its $i$-th coordinate.
	Similarly, for two indices $i,j\in\{1,\ldots,n\}$ with $i<j$ we write $x^{\neg i,\neg j}\in E^{n-2}$ for the vector that arises from $x$ by removing coordinates $i$ and $j$. Next, we define the first- and second-order difference operator applied to $f(x)=f(x_1,\ldots,x_n)$ by
	\[D_i f(x)=f(x)-f(x^{\neg i}),\]
	and
	\[D_{i,j}f(x)=D_i\big(D_j f(x)\big)=f(x)-f(x^{\neg i})-f(x^{\neg j})+f(x^{\neg i,\neg j})=D_{j,i}f(x),\]
	respectively. In other words, $D_if(x)$ measures the effect on the functional $f$ when $x_i$ is removed from $x$, and similar interpretation is valid for $D_{i,j}f(x)$. 
	
	To rephrase the normal approximation bound from \cite{LRP17} we define the following quantities:
	\begin{align*}
		B_1(f)&=\sup_{(Y,Y',Z,Z')}\EE\Big[\mathbf{1}_{\{D_{1,2}f(Y)\neq 0\}}\mathbf{1}_{\{D_{1,3}f(Y')\neq 0\}}\big(D_2f(Z)\big)^2\big(D_3f(Z')\big)^2\Big],\\
		B_2(f)&=\sup_{(Y,Z,Z')}\EE\Big[\mathbf{1}_{\{D_{1,2}f(Y)\neq 0\}}\big(D_1f(Z)\big)^2\big(D_2f(Z')\big)^2\Big],\\
		B_3(f)&=\EE\big[\left|D_1f(X)\right|^3\big],\\
		B_4(f)&=\EE\big[\left|D_1f(X)\right|^4\big],
	\end{align*} 
	where the suprema in the definitions of $B_1(f)$ and $B_2(f)$ are taken over all tuples of recombinations $(Y,Y',Z,Z')$ and $(Y,Z,Z')$, respectively, of $\{X,X',X''\}$.
	
	We are now prepared to rephrase the following normal approximation bound, which  combines Theorem~5.1 and Proposition~5.3 from \cite{LRP17} (see \cite[Remark~5.4, pp. 2007-2008]{LRP17} in a form similar to how it appeared in \cite[Lemma~2.3, p. 3066]{TTW18} but with slight modifications.
	
	Fix $n\in \NN$ and let $X_1,\ldots,X_n$ be independent, identically distributed random elements taking values in a Polish space $E$. Let $f:\cup_{k=1}^{n}E^k\to \R$ be a symmetric and measurable function. Define $W(n)=f(X_1,\ldots,X_n)$ and assume that $\EE\big[W(n)\big]=0$ and $\EE\big[\big(W(n)\big)^2\big]=1$.
	
	\begin{thm}[{\cite{LRP17}}]\label{thm:normappbound}
		Under the assumptions stated above, if $G$ denotes a standard Gaussian random variable, then
		\[\dx_W\big(W(n),G\big)\ll n\sqrt{n B_1(f)}+n\sqrt{B_2(f)}+nB_3(f)+\sqrt{nB_4(f)}.\]
	\end{thm}
	
	\section{The proof of Theorem~\ref{thm:fotetel}}\label{sec:proof}
	Our proof is based on the argument of 
	Th\"ale, Turchi and Wespi \cite{TTW18}. Let $X=(X_1,\ldots,X_n)$ be i.i.d. uniform random points from the convex disc $K\subset\R^2$ with $C^2_+$ boundary and let $X', X''$ be independent random copies of the random vector $X$.% The random vector $Z=(Z_1,\ldots,Z_n)$ is a recombination of $\{X,X',X''\}$ if $Z_i\in\{X_i,X'_i,X''_i\}$ for all $i\in\{1,\ldots,n\}$.
 
	Let
	\[f(X_1,\ldots,X_n)=\frac{A(K_n^r)-\EE[A(K_n^r)]}{\sqrt{\Var[A(K_n^r)]}},\]
	where $K_n^r=[X]_r$, and let $W(n)=f(X_1,\ldots,X_n)$.
	Note that if $x_i$, $x_j$ form an edge of $K_n^r$, then $D_{i,j}f(x)\neq 0$, so the vertices $x_i$ and $x_j$ interact. 
	However, the converse is not true as it may happen that $D_{i,j}f(x)\neq 0$ but the vertices $x_i$ and $x_j$ do not span an edge of $K_n^r$. Our argument covers this case as well. Also note, that in case three points of $X$ are on the same line, it can happen that $D_{i,j}f(x)=0$. However, the probability of this event is zero.
	
	We will use the following asymptotic lower bound for the variance of $A(K_n^r)$ from \cite{FGV22}. The matching upper bound is from \cite{FV18}:
	\begin{equation}\label{eq:lowvar}
		n^{-5/3}\ll \Var[A(K_n^r)] \ll n^{-5/3}.
	\end{equation}
	Applying Lemma~\ref{lem:uszo}, for any $\beta\in(0,\infty)$ there exists $c=c(\beta)\in(0,\infty)$ for which the random disc-polygon $[X_2,\ldots,X_n]_r$ contains with high probability the $r$-spindle floating body $K_{(c\log n/n)}^r$. If we denote this event by $A_1$, then for sufficiently large $n$, the following holds
	\begin{equation}\label{eq:B1c}
		\mathbb{P}(A_1^c)\leq (n-1)^{-\beta}\leq c_1n^{-\beta},	
	\end{equation}
	where $c_1\in (0,\infty)$ is a constant independent of $n$.
	
	We are going to estimate from above the difference operators $D_i A(K_n^r)$ and $D_{i_1,i_2}A(K_n^r)$, where $i,i_1,i_2\in\{1,\ldots,n\}$.
	
	For the sake of simplicity, we assume that $A(K)=1$. This may always be achieved by simultaneously scaling $K$ and $r$. The general statement follows simply by re-scaling. 
	
	Let $K_{n-1}^r=[X_2,\ldots,X_n]_r$. If the event $A_1$ happens and $X_1\in K_{n-1}^r$, then $A\big(K_n^r\setminus K_{n-1}^r\big)=0$. Therefore it is enough to consider the case when $X_1\in K\setminus K_{n-1}^r$. The conditional probability of this event with condition $A_1$ is (see \cite[Theorem~6.3, p. 344]{Bar08} and Lemma~\ref{lem:uszo})
	\[A(K\setminus K_{n-1}^r)\ll A(K\setminus K_{n-1})\ll A(K\setminus K_{(c\log n/n)})\ll \left(\frac{\log n}{n}\right)^{\frac 23}.\]

	Let $z\in \bd K$ be a boundary point such that its $c\log n/n$ parameter spindle visibility region contains the set of points that are in $K\setminus K_{(c\log n/n)}^r$ and which are arc-wise visible from $X_1$.  We use the following notation for the spindle visibility region of $z$:
	\[\vis_r(z,n)=\big\{x\in K\setminus K_{(c\log n/n)}^r:\exists[\overset{{{\displaystyle \frown}}}{x,z}]_r\text{ such that }[\overset{{{\displaystyle \frown}}}{x,z}]_r\cap \inti K_{(c\log n/n)}^r=\emptyset\big\}.\]
	Let $z\in \bd K$ and $L\subset K$ a spindle convex disc, and let  
	$\Delta(z,L)=\big[L\cup\{z\}\big]_r\setminus L$. In case $A_1$ happens, then
	\[\Delta\big(z,[X_2,\ldots,X_n]_r\big)=\Delta(z,K_{n-1}^r)\subset \vis_r(z,n).\]
	Using this fact and Lemma~\ref{lem:spindlevis}, we may estimate the first order differences as follows
	\begin{equation}\label{eq:first-diff}
		\left|D_1A(K_n^r)\right|\leq \sup_{z\in \bd K}A\big(\vis_r(z,n)\big) \mathbf{1}_{\{X_1\in K\setminus K^r_{(c\log n/n)}\}}\ll\frac{\log n}{n} \mathbf{1}_{\{X_1\in K\setminus K^r_{(c\log n/n)}\}}.
	\end{equation}
	
	If the event $A_1^c$ happens, then we may use the trivial estimate $|D_1 A(K_n^r)|\leq A(K)$ because the contribution of $X_1$ is at most the area of $K$. Thus,
	\begin{align*}	\EE\big[\left|D_1A(K_n^r)\right|^p\big]&=\EE\big[\mathbf{1}_{A_1}\left|D_1A(K_n^r)\right|^p\big]+\EE\big[\mathbf{1}_{A_1^c}\left|D_1A(K_n^r)\right|^p\big]\\&\ll\EE\Big[\mathbf{1}_{A_1}\left(\frac{\log n}{n}\right)^p\mathbf{1}_{\{X_1\in K\setminus K^r_{(c\log n/n)}\}}\Big]+\EE\big[\mathbf{1}_{A_1^c}A^p(K)\big]\\&\ll \left(\frac{\log n}{n}\right)^pA(K\setminus K^r_{(c\log n/n)})\ll \left(\frac{\log n}{n}\right)^{p+\frac{2}{3}}
	\end{align*}
	for all $p\in\{1,2,3,4\}$. In the third inequality we used \eqref{eq:B1c} which guarantees that the second term in the second line can be made arbitrarily small if $n$ is sufficiently large.  
	Now we can estimate the quantities $B_3(f)$ and $B_4(f)$. Using the lower bound \eqref{eq:lowvar} for the variance of $A(K_n^r)$
	we get
	\begin{align*}%\label{eq:momentum}	\notag
        \EE\big[\left|D_1f(X)\right|^p\big]&=\Var\big[A(K_n^r)\big]^{-\frac p2}\EE\big[\left|D_1A(K_n^r)\right|^p\big]\\
		&\ll n^{\frac{p}{2}\frac{5}{3}}\left(\frac{\log n}{n}\right)^{p+\frac{2}{3}}=n^{-\frac p6-\frac{2}{3}}(\log n)^{p+\frac{2}{3}}.
	\end{align*}
	In particular,
	\[nB_3(f)\ll n^{-\frac 16}(\log n)^{3+\frac{2}{3}},\]
	and
	\[\sqrt{n B_4(f)}\ll n^{-\frac 16}(\log n)^{2+\frac 13}.\]
	
	Now we turn to the second order difference operators $D_{i_1,i_2}A(K_n^r)$. Let $z\in K\setminus K^r_{(c\log n/n)}$ be a point that is not necessarily a boundary point. Let the spindle visibility region of $z$ be
	\[\vis_r(z,n)=\big\{x\in K\setminus K_{(c\log n/n)}^r:\exists[\overset{{{\displaystyle \frown}}}{x,z}]_r\text{ such that }[\overset{{{\displaystyle \frown}}}{x,z}]_r\cap \inti K_{(c\log n/n)}^r=\emptyset\big\}.\]
	Notice that if $\vis_r(X_1,n)$ and $\vis_r(X_2,n)$ are disjoint, then 
	$D_{1,2}A(K_n^r)=0$. Let $Y,Y', Z$ and $Z'$ be recombinations of $\{X,X',X''\}$, and let $A_2$ denote the event
	\[K^r_{(c\log n/n)}\subseteq\bigcap_{W\in\{Y,Y',Z,Z'\}}[W_4,\ldots,W_n]_r.\]
	Then the probability of the complement of $A_2$ is also small
	\begin{equation}\label{eq:a2}
		\mathbb{P}(A_2^c)\leq c_2n^{-\beta},
	\end{equation}
	where $c_2\in(0,\infty)$ is a constant independent from $n$.
	
	If the event $A_2$ happens, then it follows from \eqref{eq:first-diff} that
	\[\big(D_{i}A(K_n^r)\big)^2\ll \left(\frac{\log n}{n}\right)^2,\]
	furthermore, using the lower bound \eqref{eq:lowvar} we get that
	\[\big(D_if(V)\big)^2%=\frac{\big(D_{i}A(K_n^r)\big)^2}{\Var\big[A(K_n^r)\big]}
	\ll \left(\frac{\log n}{n}\right)^2n^{\frac 53}=n^{-\frac 13}(\log n)^2\]
	for $i\in \{1,2,3\}$ and $V\in\{X,X'\}$. We note that if $A_2$ happens, then 
	\begin{align*}
		\{D_{1,2}f(Y)\neq 0\}&\subseteq \{Y_1\in K\setminus K^r_{(c\log n/n)}\}\cap \{Y_2\in K\setminus K^r_{(c\log n/n)}\}\\ &\qquad\cap \{\vis_r(Y_1,n)\cap\vis_r(Y_2,n)\neq \emptyset\}\\&\subseteq \{Y_1\in K\setminus K^r_{(c\log n/n)}\}\cap \bigg\{Y_2\in \bigcup_{x\in \vis_r(Y_1,n)}\vis_r(x,n)\bigg\}.
	\end{align*}
	
	If $A_2$ happens then $[Y_4,\ldots, Y_n]_r$ already contains $K^r_{(c\log n/n)}$. Thus, $D_{1,2}f(Y)$ is nonzero if $Y_1,Y_2\in K\setminus K^r_{(c\log n/n)}$ and the spindle visibility regions of $Y_1$ and $Y_2$ are not disjoint, which means that they "see each other with circular arcs". Then $Y_1,Y_2$ either contribute with an edge to $K_n^r$, or removing $Y_1$ the point $Y_2$ becomes a vertex of $K_n^r$ (or vice versa).
	
	Similar conditions are satisfied for $D_{1,3}f(Y')$. Therefore,
	\begin{align*}
		\EE&\big[\mathbf{1}_{\{D_{1,2}f(Y)\neq 0\}}\mathbf{1}_{A_2}\big]\leq\pr \big(Y_1\in K\setminus K^r_{(c\log n/n)}\big)\times\\&\qquad\times \pr\Big(Y_2\in \bigcup_{x\in \vis_r(Y_1,n)}\vis_r(x,n)\,\Big\vert\, Y_1\in K\setminus K^r_{(c\log n/n)}\Big)\\
		&\leq \pr\big(Y_1\in K\setminus K^r_{(c\log n/n)}\big)\sup_{z\in K\setminus K^r_{(c\log n/n)}}\!\!\!\pr \Big(Y_2\in \!\!\! \bigcup_{x\in \vis_r(z,n)}\!\!\!\vis_r(x,n)\Big)\\
		&=A\big(K\setminus K^r_{(c\log n/n)}\big)\sup_{z\in K\setminus K^r_{(c\log n/n)}}A\Big(\bigcup_{x\in \vis_r(z,n)}\!\!\!\vis_r(x,n)\Big).
	\end{align*}
	Since $\vis_r(x,n)$ is the union of all area $(c\log n/n)$ disc-caps that contain $x\in \vis_r(z,n)$, it follows from \eqref{eq:szendvics} that
	\[{\rm diam\,}\Big(\bigcup_{x\in \vis_r(z,n)}\!\!\!\vis_r(x,n)\Big)\ll \left(\frac{\log n}{n}\right)^{\frac13}\]
    for any $z\in K\setminus K^r_{(c\log n/n)}$. Thus, \cite[pp. 2149-2150]{Rei03} and Lemma~\ref{lem:spindlevis} yield that
	\[\Delta(n):=\sup_{z\in K\setminus K^r_{(c\log n/n)}}A\Big(\bigcup_{x\in \vis_r(z,n)}\!\!\!\vis_r(x,n)\Big)\ll\frac{\log n}{n}.\]
	Furthermore, for $A_2^c$, we may estimate the indicator functions by one and the difference operators by $A(K)$. 
	Since $\pr(A_2^c)$ is small (see \eqref{eq:a2}),
	\[B_2(f)\ll n^{-\frac 23}(\log n)^4A(K\setminus K^r_{(c\log n/n)})\Delta(n)\ll %n^{-\frac 23}(\log n)^4\left(\frac{\log n}{n}\right)^{\frac 23}\frac{\log n}{n}=
	n^{-\frac 73}(\log n)^{5+\frac 23}.\]
	The quantity $B_1(f)$ can be estimated similarly. First assume that $Y_1=Y_1'$. Under the assumption that $A_2$ happens, we get that
	%\begin{align*}
	%	\{D_{1,2}f(Y)\neq 0\}\cap \{D_{1,3}f(Y')\neq 0\}&\subseteq \big\{\{Y_1,Y_2,Y_3'\}\subseteq K\setminus K_{(c\log n/n)}\}\big\}\\&\qquad\cap \{{\rm Vis}_{Y_1}(n)\cap{\rm Vis}_{Y_2}(n)\neq \emptyset\}\\&\qquad \cap \{{\rm Vis}_{Y_1}(n)\cap{\rm Vis}_{Y_3'}(n)\neq \emptyset\}\\&\subseteq \{Y_1\in K\setminus K_{(c\log n/n)}\}\\&\qquad \cap  \bigg\{\{Y_2,Y_3'\}\subseteq \bigcup_{x\in {\rm Vis}_{Y_1}(n)}{\rm Vis}_x(n)\bigg\}.
	%\end{align*}
	\begin{align*}
		\{D_{1,2}&f(Y)\neq 0\}\cap \{D_{1,3}f(Y')\neq 0\}\subseteq\\[5pt]&\subseteq \big\{\{Y_1,Y_2,Y_3'\}\subseteq K\setminus K^r_{(c\log n/n)}\}\big\} \cap \{\vis_r(Y_1,n)\cap\vis_r(Y_2,n)\neq \emptyset\}\\&\qquad \cap \{\vis_r(Y_1,n)\cap\vis_r(Y_3',n)\neq \emptyset\}\subseteq\\[5pt]&\subseteq \{Y_1\in K\setminus K^r_{(c\log n/n)}\} \cap  \bigg\{\{Y_2,Y_3'\}\subseteq \!\!\! \bigcup_{x\in \vis_r(Y_1,n)}\!\!\!\vis_r(x,n)\bigg\}.
	\end{align*}
	By the above argument
	\begin{align*}
		\EE\big[\mathbf{1}_{\{D_{1,2}f(Y)\neq 0\}}&\mathbf{1}_{\{D_{1,3}f(Y')\neq 0\}}\mathbf{1}_{A_2}\big]\leq \pr\big(Y_1\in K\setminus K^r_{(c\log n/n)}\big)\times\\&\quad\times\sup_{z\in K\setminus K^r_{(c\log n/n)}}\pr \Big(\{Y_2,Y_3'\}\subseteq \!\!\!\bigcup_{x\in \vis_r(z,n)}\!\!\!\vis_r(x,n)\Big)\\
		&\leq A(K\setminus K^r_{(c\log n/n)})\big(\Delta(n)\big)^2.
	\end{align*}
	If $Y_1\neq Y_1'$, then we get a smaller order of magnitude since in that case we have an extra factor $A\big(K\setminus K^r_{(c\log n/n)}\big)$ by the independence. Thus,
	\begin{align*}	B_1(f)&=\sup_{(Y,Y',Z,Z')}\EE\Big[\mathbf{1}_{\{D_{1,2}f(Y)\neq 0\}}\mathbf{1}_{\{D_{1,3}f(Y')\neq 0\}}\big(D_2f(Z)\big)^2\big(D_3f(Z')\big)^2\Big]\\&\ll \EE\Big[\mathbf{1}_{A_2} \mathbf{1}_{\{D_{1,2}f(Y)\neq 0\}}\mathbf{1}_{\{D_{1,3}f(Y')\neq 0\}}n^{-\frac 23}(\log n)^4\Big]+\EE\Big[\mathbf{1}_{A_2^c}A^2(K)\Big]\\&
		\ll n^{-\frac 23}(\log n)^4A(K\setminus K^r_{(c\log n/n)})\big(\Delta(n)\big)^2\ll n^{-\frac{10}{3}}(\log n)^{6+\frac 23}.
	\end{align*}
	Now, we can estimate the other two terms in Theorem~\ref{thm:normappbound}.
	\begin{align*}
		n\sqrt{n B_1(f)}&\ll n\sqrt{n^{-\frac{7}{3}}(\log n)^{6+\frac 23}}=n^{-\frac{1}{6}}(\log n)^{3+\frac 13}.\\
		n\sqrt{B_2(f)}&\ll n \sqrt{n^{-\frac 73}(\log n)^{5+\frac 23}}=n^{-\frac{1}{6}}(\log n)^{2+\frac{5}{6}}.
	\end{align*}
	Finally, substituting our estimates in Theorem~\ref{thm:normappbound} we get that
	\begin{align*}
		\dx_W\left(W(n),G\right)&\ll n\sqrt{n B_1(f)}+n\sqrt{B_2(f)}+nB_3(f)+\sqrt{nB_4(f)}\\&\ll n^{-\frac 16}\big((\log n)^{3+\frac 13}+(\log n)^{2+\frac 56}+(\log n)^{3+\frac 23}+(\log n)^{2+\frac 13}\big)\\&\ll n^{-\frac 16}(\log n)^{3+\frac 23},	
	\end{align*}
	where $G$ is a random variable with standard normal distribution. Since the Wasserstein distance of the random variable $W(n)$ and $G$ tends to zero as $n\to\infty$, thus
	\[W(n)=\frac{A(K_n^r)-\EE[A(K_n^r)]}{\sqrt{\Var[A(K_n^r)]}}\xrightarrow{\mathcal{D}}G\sim\mathcal{N}(0,1),\]
	which finishes the proof of Theorem~\ref{thm:fotetel}.

	\section{Acknowledgments}
	F. Fodor was supported by the National Research, Development and Innovation Office – NKFIH K134814 grant.
	
	This research was supported by project TKP2021-NVA-09. Project no. TKP2021-NVA-09 has been implemented with the support provided by the Ministry of Innovation and Technology of Hungary from the National Research, Development and Innovation Fund, financed under the TKP2021-NVA funding scheme.


\begin{bibdiv}
		\begin{biblist}
			
			\bib{Bar08}{article}{
				author={B\'{a}r\'{a}ny, Imre},
				title={Random points and lattice points in convex bodies},
				journal={Bull. Amer. Math. Soc. (N.S.)},
				volume={45},
				date={2008},
				number={3},
				pages={339--365},
				issn={0273-0979},
				%	review={\MR{2402946}},
				%	doi={10.1090/S0273-0979-08-01210-X},
			}
			
			\bib{BD97}{article}{
				author={B\'{a}r\'{a}ny, Imre},
				author={Dalla, Leoni},
				title={Few points to generate a random polytope},
				journal={Mathematika},
				volume={44},
				date={1997},
				number={2},
				pages={325--331},
				issn={0025-5793},
				%review={\MR{1600549}},
				%doi={10.1112/S0025579300012638},
			}
			
			\bib{BL88}{article}{
				author={B\'{a}r\'{a}ny, I.},
				author={Larman, D. G.},
				title={Convex bodies, economic cap coverings, random polytopes},
				journal={Mathematika},
				volume={35},
				date={1988},
				number={2},
				pages={274--291},
				issn={0025-5793},
				%review={\MR{986636}},
				%doi={10.1112/S0025579300015266},
			}
			
			\bib{BR10}{article}{
				author={B\'{a}r\'{a}ny, Imre},
				author={Reitzner, Matthias},
				title={Poisson polytopes},
				journal={Ann. Probab.},
				volume={38},
				date={2010},
				number={4},
				pages={1507--1531},
				issn={0091-1798},
				%review={\MR{2663635}},
				%doi={10.1214/09-AOP514},
			}
			
			\bib{BRT21}{article}{
				author={Besau, Florian},
				author={Rosen, Daniel},
				author={Th\"{a}le, Christoph},
				title={Random inscribed polytopes in projective geometries},
				journal={Math. Ann.},
				volume={381},
				date={2021},
				number={3-4},
				pages={1345--1372},
				issn={0025-5831},
				%review={\MR{4333417}},
				%doi={10.1007/s00208-021-02257-9},
			}
			
			\bib{BL07}{article}{
				author={Bezdek, K\'{a}roly},
				author={L\'{a}ngi, Zsolt},
				author={Nasz\'{o}di, M\'{a}rton},
				author={Papez, Peter},
				title={Ball-polyhedra},
				journal={Discrete Comput. Geom.},
				volume={38},
				date={2007},
				number={2},
				pages={201--230},
				issn={0179-5376},
				%	review={\MR{2343304}},
				%	doi={10.1007/s00454-007-1334-7},
			}
			
			\bib{Cha08}{article}{
				author={Chatterjee, Sourav},
				title={A new method of normal approximation},
				journal={Ann. Probab.},
				volume={36},
				date={2008},
				number={4},
				pages={1584--1610},
				issn={0091-1798},
				%review={\MR{2435859}},
				%doi={10.1214/07-AOP370},
			}
			
			\bib{FGV22}{article}{
				author={Fodor, Ferenc},
				author={Grünfelder, Balázs},
				author={V\'{\i}gh, Viktor},
				title={Variance bounds for disc-polygons},
				journal={Doc. Math.},
				date={2022},
				number={27},
				pages={1015--1029},
			}
			
			\bib{FKV14}{article}{
				author={Fodor, F.},
				author={Kevei, P.},
				author={V\'{i}gh, V.},
				title={On random disc polygons in smooth convex discs},
				journal={Adv. in Appl. Probab.},
				volume={46},
				date={2014},
				number={4},
				pages={899--918},
				issn={0001-8678},
				%review={\MR{3290422}},
				%doi={10.1239/aap/1418396236},
			}
			
			\bib{FV18}{article}{
				author={Fodor, Ferenc},
				author={V\'{\i}gh, Viktor},
				title={Variance estimates for random disc-polygons in smooth convex
					discs},
				journal={J. Appl. Probab.},
				volume={55},
				date={2018},
				number={4},
				pages={1143--1157},
				issn={0021-9002},
				%review={\MR{3899933}},
				%doi={10.1017/jpr.2018.76},
			}
			
			\bib{LRP17}{article}{
				author={Lachi\`eze-Rey, Rapha\"{e}l},
				author={Peccati, Giovanni},
				title={New Berry-Esseen bounds for functionals of binomial point
					processes},
				journal={Ann. Appl. Probab.},
				volume={27},
				date={2017},
				number={4},
				pages={1992--2031},
				issn={1050-5164},
				%review={\MR{3693518}},
				%doi={10.1214/16-AAP1218},
			}
			
			\bib{LSY19}{article}{
				author={Lachi\`eze-Rey, Rapha\"{e}l},
				author={Schulte, Matthias},
				author={Yukich, J. E.},
				title={Normal approximation for stabilizing functionals},
				journal={Ann. Appl. Probab.},
				volume={29},
				date={2019},
				number={2},
				pages={931--993},
				issn={1050-5164},
				%review={\MR{3910021}},
				%doi={10.1214/18-AAP1405},
			}
			
			\bib{L}{article}{
				author={L\'{a}ngi, Zsolt},
				author={Nasz\'{o}di, M\'{a}rton},
				author={Talata, Istv\'{a}n},
				title={Ball and spindle convexity with respect to a convex body},
				journal={Aequationes Math.},
				volume={85},
				date={2013},
				number={1-2},
				pages={41--67},
				issn={0001-9054},
				%review={\MR{3028202}},
				%doi={10.1007/s00010-012-0160-z},
			}
			
			\bib{MMO19}{book}{
				author={Martini, Horst},
				author={Montejano, Luis},
				author={Oliveros, D\'{e}borah},
				title={Bodies of constant width},
				%note={An introduction to convex geometry with applications},
				publisher={Birkh\"{a}user/Springer, Cham},
				date={2019},
				pages={xi+486},
				%isbn={978-3-030-03866-3},
				%isbn={978-3-030-03868-7},
				%review={\MR{3930585}},
				%doi={10.1007/978-3-030-03868-7},
			}
			
			\bib{Rei03}{article}{
				author={Reitzner, Matthias},
				title={Random polytopes and the Efron-Stein jackknife inequality},
				journal={Ann. Probab.},
				volume={31},
				date={2003},
				number={4},
				pages={2136--2166},
				issn={0091-1798},
				%review={\MR{2016615}},
				%doi={10.1214/aop/1068646381},
			}
			
			\bib{Rei05}{article}{
				author={Reitzner, Matthias},
				title={Central limit theorems for random polytopes},
				journal={Probab. Theory Related Fields},
				volume={133},
				date={2005},
				number={4},
				pages={483--507},
				issn={0178-8051},
				%review={\MR{2197111}},
				%doi={10.1007/s00440-005-0441-8},
			}
			
			\bib{R10}{article}{
				author={Reitzner, Matthias},
				title={Random polytopes},
				conference={
					title={New perspectives in stochastic geometry},
				},
				book={
					publisher={Oxford Univ. Press, Oxford},
				},
				date={2010},
				pages={45--76},
				%  review={\MR{2654675}},
			}
			
			\bib{RS63}{article}{
				author={R\'{e}nyi, A.},
				author={Sulanke, R.},
				title={\"{U}ber die konvexe H\"{u}lle von $n$ zuf\"{a}llig gew\"{a}hlten Punkten},
				language={German},
				journal={Z. Wahrscheinlichkeitstheorie und Verw. Gebiete},
				volume={2},
				date={1963},
				pages={75--84 (1963)},
				%review={\MR{0156262}},
				%doi={10.1007/BF00535300},
			}
			
			\bib{RS64}{article}{
				author={R\'{e}nyi, A.},
				author={Sulanke, R.},
				title={\"{U}ber die konvexe H\"{u}lle von $n$ zuf\"{a}llig gew\"{a}hlten Punkten. II},
				language={German},
				journal={Z. Wahrscheinlichkeitstheorie und Verw. Gebiete},
				volume={3},
				date={1964},
				pages={138--147 (1964)},
				%review={\MR{169139}},
				%doi={10.1007/BF00535973},
			}
			
			
			\bib{Schneider}{book}{
				author={Schneider, Rolf},
				title={Convex bodies: the Brunn-Minkowski theory},
				series={Encyclopedia of Mathematics and its Applications},
				volume={151},
				edition={Second expanded edition},
				publisher={Cambridge University Press, Cambridge},
				date={2014},
				pages={xxii+736},
				isbn={978-1-107-60101-7},
				%review={\MR{3155183}},
			}
			
			\bib{Sch18}{article}{
				author={Schneider, Rolf},
				title={Discrete aspects of stochastic geometry},
				conference={
					title={Handbook of discrete and computational geometry, 3rd ed.},
				},
				book={
					%series={CRC Press Ser. Discrete Math. Appl.},
					publisher={CRC, Boca Raton, FL},
				},
				date={2018},
				pages={299--329},
				%   review={\MR{1730165}},
			}
			
			\bib{SW90}{article}{
				author={Sch\"{u}tt, Carsten},
				author={Werner, Elisabeth},
				title={The convex floating body},
				journal={Math. Scand.},
				volume={66},
				date={1990},
				number={2},
				pages={275--290},
				issn={0025-5521},
				%review={\MR{1075144}},
				%doi={10.7146/math.scand.a-12311},
			}
			
			\bib{Ste86}{book}{
				author={Stein, Charles},
				title={Approximate computation of expectations},
				series={Institute of Mathematical Statistics Lecture Notes---Monograph
					Series},
				volume={7},
				publisher={Institute of Mathematical Statistics, Hayward, CA},
				date={1986},
				pages={iv+164},
				isbn={0-940600-08-0},
				%review={\MR{882007}},
			}
			
			\bib{T18}{article}{
				author={Th\"{a}le, Christoph},
				title={Central limit theorem for the volume of random polytopes with
					vertices on the boundary},
				journal={Discrete Comput. Geom.},
				volume={59},
				date={2018},
				number={4},
				pages={990--1000},
				issn={0179-5376},
				%review={\MR{3802312}},
				%doi={10.1007/s00454-017-9862-2},
			}
			
			\bib{TTW18}{article}{
				author={Th\"{a}le, Christoph},
				author={Turchi, Nicola},
				author={Wespi, Florian},
				title={Random polytopes: central limit theorems for intrinsic volumes},
				journal={Proc. Amer. Math. Soc.},
				volume={146},
				date={2018},
				number={7},
				pages={3063--3071},
				issn={0002-9939},
				%review={\MR{3787367}},
				%doi={10.1090/proc/14000},
			}
			
			\bib{V09}{book}{
				author={Villani, C\'{e}dric},
				title={Optimal transport},
				series={Grundlehren der mathematischen Wissenschaften [Fundamental
					Principles of Mathematical Sciences]},
				volume={338},
				note={Old and new},
				publisher={Springer-Verlag, Berlin},
				date={2009},
				%pages={xxii+973},
				%isbn={978-3-540-71049-3},
				%review={\MR{2459454}},
				%doi={10.1007/978-3-540-71050-9},
			}
			
			\bib{Vu05}{article}{
				author={Vu, V. H.},
				title={Sharp concentration of random polytopes},
				journal={Geom. Funct. Anal.},
				volume={15},
				date={2005},
				number={6},
				pages={1284--1318},
				issn={1016-443X},
				%review={\MR{2221249}},
				%doi={10.1007/s00039-005-0541-8},
			}
			
		\end{biblist}	
		
	\end{bibdiv}
\end{document}